\documentclass[a4paper,12pt]{amsart}
\usepackage{amsmath,amssymb,amsfonts,amsthm}
\usepackage{latexsym,mathrsfs}
\usepackage{graphicx}
\usepackage[all]{xy}
\input xypic
\usepackage{color}
\usepackage[backref=false]{hyperref}
\hypersetup{colorlinks=true,linkcolor=red,citecolor=blue}
\usepackage[OT2,T1]{fontenc}
\usepackage[utf8]{inputenc}
\usepackage{lmodern}
\usepackage{microtype}
\usepackage{enumerate}
\usepackage{tikz-cd}
\usepackage{bm}
\usepackage{ulem}
\usepackage{mathtools}
\usepackage{tensor}
\theoremstyle{plain}
\newtheorem{theorem}[subsection]{{\bf Theorem}}
\newtheorem{corollary}[subsection]{{\bf Corollary}}
\newtheorem{proposition}[subsection]{{\bf Proposition}}
\newtheorem{lemma}[subsection]{{\bf Lemma}}
\theoremstyle{definition}

\theoremstyle{remark}

\numberwithin{equation}{subsection}

\DeclareMathOperator{\Aut}{Aut}
\DeclareMathOperator{\Out}{Out}
\DeclareMathOperator{\Inn}{Inn}

\DeclareBoldMathCommand{\bbot}{\bot}

\DeclareSymbolFont{cyrletters}{OT2}{wncyr}{m}{n}
\DeclareMathSymbol{\Sha}{\mathalpha}{cyrletters}{"58}

\DeclareMathOperator{\ad}{ad}
\DeclareMathOperator{\QHom}{QHom}
\DeclareMathOperator{\MQHom}{MQHom}

\DeclareMathOperator{\Poly}{Poly}
\DeclareMathOperator{\QPoly}{QPoly}

\DeclareMathOperator{\Conj}{Conj}
\DeclareMathOperator{\Comm}{Comm}
\DeclareMathOperator{\LComm}{LComm}
\DeclareMathOperator{\RComm}{RComm}
\DeclareMathOperator{\Pol2}{Pol_2}
\DeclareMathOperator{\CAT}{CAT}
\DeclareMathOperator{\grph}{\mathfrak{gr}}

\begin{document}
\title[Middle quasi-homomorphisms]{Quasi-affine and quasi-quadratic maps of groups with non-abelian targets}
\author{Primo\v z Moravec}
\address{{
Faculty of  Mathematics and Physics, University of Ljubljana,
and Institute of Mathematics, Physics and Mechanics,
Slovenia}}
\email{primoz.moravec@fmf.uni-lj.si}
\subjclass[2020]{}
\thanks{ORCID: \url{https://orcid.org/0000-0001-8594-0699}. The author acknowledges the financial support from the Slovenian Research and Inovation Agency (ARIS), research core funding No. P1-0222, and project No. J1-50001.}
\date{\today}
\begin{abstract}
\noindent
It is shown that the middle quasi-homomorphisms of Fujiwara and Kapovich are precisely constant perturbations of quasi-homomorphisms. Quasi-polynomial maps are defined and their constructibility is explored. In particular, it is shown that a large class of quasi-quadratic maps into torsion-free hyperbolic groups is rigid with respect to bounded perturbations.
\end{abstract}
\maketitle

\section{Introduction}
\label{s:intro}

\noindent
Let $G$ and $H$ be groups. Let $d$ be a left-invariant metric on $H$. A map $\phi:G\to H$ is a {\it quasi-homomorphism} if
$$\sup_{x,y\in G} d(\phi(xy),\phi(x)\phi(y))<\infty.$$
The concept of quasi-homomorphisms goes back to Ulam \cite{Ula60}. Quasi-homomorphisms with commutative target groups $H$ (also known as {\it quasi-morphisms}) have been well studied, especially in connection with bounded cohomology and stable commutator length. An important line of research addresses rigidity of quasi-morphisms with respect to their distance to homomorphisms. 

We assume throughout that $H$ is discrete with left-invariant metric $d$, and not necessarily commutative. Given a map $\phi: G\to H$, define its {\it (left) defect set} by
$$D(\phi)=\{ \phi(y)^{-1}\phi(x)^{-1}\phi (xy)\mid x,y\in G\}.$$
Then it is clear that $\phi$ is a quasi-homomorphism if and only the set $D(\phi)$ is bounded, that is, finite with order bounded by an unspecified constant. Fujiwara and Kapovich showed in their groundbreaking paper \cite{FK16} that all quasi-homomorphisms with discrete targets are, in a sense, constructible in a standard way. It is also exhibited that quasi-homomorphisms into torsion-free hyperbolic groups are rigid with respect to bounded perturbations.

Following Ozawa's suggestion, Fujiwara and Kapovich \cite{FK16} defined the concept of {\it middle quasi-homomorphism} $\phi: G\to H$ by requiring the {\it middle defect set} 
$$M(\phi)=\{ \phi(x)^{-1}\phi(xy)\phi(y)^{-1}\mid x,y\in G\}$$ 
of $\phi$ to be bounded. Note that if the target group $H$ is abelian, then this notion coincides with that of quasi-homomorphisms. Middle quasi-homomorphisms are special instances of the so-called algebraic quasi-homomorphisms and geometric quasi-homomorphisms defined in {\it loc. cit}. Both these classes of maps are closed for bounded perturbations, whereas bounded perturbations of (middle) quasi-homomorphisms do not necessarily yield (middle) quasi-homomorphisms again. 

The exact relationship between quasi-homomorphisms and middle quasi-homomorphisms was left open in \cite{FK16}.
Our first result characterizes middle quasi-homomorphisms in terms of quasi-homomorphisms:

\begin{theorem}
    \label{thm:qhom}
    Middle quasi-homomorphisms $G\to H$ are precisely constant perturbations of quasi-homomorphisms. In particular, every unital middle quasi-homomorphism is a quasi-homomorphism.
\end{theorem}

In view of Theorem \ref{thm:qhom} one can therefore think of the middle quasi-homomorphisms as {\it affine quasi-homomorphisms}. In particular, every middle quasi-homomorphism is only a constant away from being a quasi-homomorphism.

In \cite[Theorem 9.6]{FK16}, a unital map $F_2\to F_2$ is constructed that is claimed to be a middle quasi-homomorphism, which is not a quasi-homomorphism. Theorem \ref{thm:qhom} contradicts this. We elaborate on this in Section \ref{s:middle}.

As already mentioned above, bounded perturbations of a quasi-homomorphism, even by a given constant, do not necessarily yield quasi-homomorphisms. In Section \ref{s:perturbing}, we consider the question which constant perturbations of a given quasi-homomorphism $\phi:G\to H$ is again a quasi-homomorphism. The set $\Pi_\phi$ of all such constants is a subgroup of $H$. We characterize this group in the case when $H$ is torsion-free hyperbolic (Proposition \ref{prop:hyperbolic}). Our main result in this direction is a description of $\Pi_\phi$ in terms of the commensurator of the graph $\grph(\phi)$ of $\phi$ within $G\times H$. For symmetric quasi-homomorphisms, it reads as follows:

\begin{theorem}
    \label{thm:Piagain}
    If $\phi:G\to H$ is a symmetric quasi-homomorphism, then ${\rm proj}_{H}(\Comm_{G\times H}\grph(\phi)\cap 1\times H)=\Pi_\phi$.
\end{theorem}

The argument goes by first characterizing graphs of middle quasi-homomorphisms as right translates of left quasi-subgroups of $G\times H$, and then looking at perturbations of graphs of quasi-homomorphisms. 

In the second part of the paper we describe how  middle quasi-homomorphisms are closely related to quasi-polynomial maps. The notion of polynomial maps on groups goes back to Leibman \cite{Lei98,Lei02}. An approximate version of these was recently introduced by Jamneshan and Thom \cite{JT24} under the name uniform $\epsilon$-polynomials. In the coarse setting, we define a map $\phi:G\to H$ to be a quasi-polynomial of degree $d$ if the range $P_d(\phi)$ of the $(d+1)$-iterated non-commutative difference map is bounded, see Section \ref{s:qpoly} for a concise definition. With this definition, quasi-polynomials of degree 0 are precisely the bounded maps, and degree 1 quasi-polynomials are precisely middle quasi-homomorphisms.

Our main aim is to develop the notion of constructibility for quasi-polynomials that resembles that of quasi-homomorphisms developed by Fujiwara and Kapovich \cite{FK16}. We restrict ourselves to quasi-polynomials of degree 2, also called the quasi-quadratic maps; the theory extends to higher degree quasi-polynomials, with straightforward but tedious technical adjustments. Compared to degree 1, there are two important differences. First of all, one needs to restrict to quasi-quadratic maps satisfying some finiteness conditions in the spirit of Hrushovski \cite{Hru12}. These maps are called the normal quadratic maps. Their fundamental property is that $\phi(G)\cup\phi (G)^{-1}$ is covered by finitely many cosets of the centralizer of the quadratic defect subgroup, and that the balls generated by the corresponding transversal eventually shrink. All quasi-homomorphisms satisfy these conditions for fairly apparent reasons. The second deviation point is to, rather than considering the map $\phi:G\to H$, induce a normal quasi-quadratic map $\hat{\phi}: \Pol2(G)\to H$ in a canonical way and then prove a constructibility result for it. Here, $\Pol2 G$ is a universal construction introduced and described in \cite{JT24}. 

Our representability result reads as follows:

\begin{theorem}
    \label{thm:normalpoly}
    Every normal quasi-quadratic map $\phi: G\to H$ is constructible in the following sense: 
    \begin{enumerate}
        \item There is an induced normal quasi-quadratic map $\hat{\phi}:\Pol2 G\to H$,
        \item There exist a subgroup $G_o$ of $\Pol2 G$ with $|\Pol2 G:G_o|<\infty$, a normal quasi-quadratic map $\phi_o:G_o\to H_o\le H$ within a finite distance from $\hat{\phi}$ (restricted to $G_o$), and there is a finitely generated central subgroup $Z_o$ of $H_o$, such that the induced map $\tilde{\phi}_o:G_o\to H_o/Z_o$ is quadratic.
    \end{enumerate}
\end{theorem}

This has several consequences, similar as those for the situation of quasi-homomorphisms described in \cite{FK16}. We only touch upon the case where the target group $H$ is torsion-free hyperbolic (Proposition \ref{prop:normalhyperb}). The other cases of \cite{FK16}, such as $\CAT(0)$ targets, targets that are mapping class groups or groups acting on trees, could be dealt with in a similar way.

As shown by Fujiwara and Kapovich \cite[Theorem 4.4]{FK16}, the unbounded quasi-homomorphisms are rigid with respect to bounded perturbations. We show the same for unbounded normal quasi-quadratic maps (in fact, quasi-quadratic maps satisfying a slightly less restrictive condition):

\begin{theorem}
    \label{thm:distancenearly}
    Let $H$ be a torsion-free hyperbolic group. Let $\phi_1$ and $\phi_2$ be nearly normal quasi-quadratic maps from $G$ to $H$. Suppose that $d(\phi_1,\phi_2)<\infty$. Then either both $\phi_1$, $\phi_2$ are bounded, or they both map into the same cyclic subgroup of $H$, or $\phi_1=\phi_2$. 
\end{theorem}

The crucial step in proving Theorem \ref{thm:distancenearly} is to observe that if $\phi:\mathbb{Z}\to H$ is a unital quadratic map, then $\phi(1)$ and $\phi(2)$ satisfy a non-trivial relation in $H$.

 Another aspect of middle quasi-homomorphisms and quasi-quadratic maps is to consider their actions on multiplicative quadruples introduced by Gowers \cite{Gow98}. Define $\mathfrak{M}(G)$ to be the set of all quadruples $(x_1,x_2,x_3,x_4)\in G^{\times 4}$ with the property that $x_1x_2^{-1}x_3x_4^{-1}=1$. Given $\phi:G\to H$, we have a map $\mu_\phi:\mathfrak{M}(G)\to H$ given by $\mu_\phi(x_1,x_2,x_3,x_4)=\phi(x_1)\phi(x_2)^{-1}\phi(x_3)\phi(x_4)^{-1}$. We then have:
 \begin{theorem}
    \label{thm:multq}
    Let $\phi:G\to H$ be a map.
    \begin{enumerate}
        \item $\phi$ is a middle quasi-homomorphism if and only if $\mu_\phi$ is bounded.
        \item Let $\phi(1)=1$. Then $\phi$ is quasi-quadratic if and only if $\mu_\phi$ is $G$-equivariant up to bounded error.
    \end{enumerate}
 \end{theorem}

\section{Middle quasi-homomorphisms}
\label{s:middle}

\noindent
\paragraph{\bf Notations.}
 Let $G$ be a group, and $H$ be a discrete group with a proper left-invariant metric $d$. Given $h\in H$, we write $\lVert h\rVert=d(h,1)$. For $\phi,\psi : G\to H$ we define
$$d(\phi,\psi) =\sup_{x\in G} d(\phi(x),\psi(x))$$
and $\lVert\phi\rVert=d(\phi, 1)$, where $1:G\to H$ s the map sending all $g\in G$ to $1$. 

Let $S$ be a bounded subset of $H$ and $a,b\in H$. As in \cite{FK16} we write $a\sim_S b$ iff there exists $s\in S$ with $a=bs$. Note that if $a\sim_S b$, then $d(a,b)\le\max_{s\in S} \lVert s\rVert$.

Given two subsets $A$ and $B$ of a group, we denote $AB=\{ ab\mid a\in A,b\in B\}$ and $A^{-1}=\{ a^{-1}\mid a\in A\}$. A set $A$ is said to be symmetric if $A=A^{-1}$. Given a positive integer $n$ we set $A^n$ to be the set of all products of the form $a_1a_2\cdots a_n$, where $a_i\in A$, and $A^{-n}=(A^{-1})^n$.

Let $a,b$ be elements of a group. Then we write $a^b=b^{-1}ab$ and $[a,b]=a^{-1}a^b$. We denote the inner automorphism of the group induced by conjugation with $a$ from the left by $\ad a$, so $(\ad a)(x)=axa^{-1}$.

\medskip

\paragraph{\bf Quasi-homomorphisms.}
We collect here some basic properties of quasi-homomorphisms. The following is immediate:

\begin{lemma}
    \label{lem:quasi}
    Let $\phi:G\to H$ be a quasi-homomorphism with $D=D(\phi)$. Let $x,y\in G$.
    \begin{enumerate}
        \item $\phi(xy)\sim_D\phi(x)\phi(y)$.
        \item $\phi(x)^{-1}\sim_{D^2}\phi(x^{-1})$.
    \end{enumerate}
\end{lemma}

The set of all quasi-homomorphisms $G\to H$ will be denoted by $\QHom(G,H)$.
Alternatively, set 
$$D^*(\phi)=\{\phi(x)\phi(y)\phi(xy)^{-1}\mid x,y\in G\}.$$
This is the {\it right defect set} of $\phi$.
Then it can be shown that $\phi\in\QHom(G,H)$ if and only if $D^*(\phi)$ is bounded \cite[Proposition 3.1.6]{Heu19}. 

We call the group
$\Delta_\phi =\langle D(\phi)\rangle$
to be the {\it defect subgroup} of $\phi$. For $R>0$ and $S\subseteq H$, we denote the $R$-neighborhood of $S$ in $H$ by $\mathcal{N}_R(S)$.

\begin{lemma}[Lemma 2.3 and Corollary 3.1 of \cite{FK16}]
    \label{lem:delta}
    Let the map $\phi:G\to H$ be a quasi-homomorphism. Then:
    \begin{enumerate}
        \item $\phi (G)$ is contained in the normalizer $N_H(\Delta_\phi)$ of $\Delta_\phi$ in $H$.
        \item  There exists $R>0$ such that $\phi(G)\subseteq \mathcal{N}_R(C_H(\Delta _\phi))$.
        \item There exists a finite subset $Y$ of $H$ such that $\phi(G)\cup\phi(G)^{-1}$ is covered by the union of cosets $C_H(\Delta _\phi)y$, where $y\in Y$.
    \end{enumerate}  
\end{lemma}

We remark that (3) follows directly from Lemma \ref{lem:delta} (2) and Lemma \ref{lem:quasi} (2). Also, it is not difficult to see that the above result implies that if $\phi$ is a quasi-homomorphism, then $\Delta_\phi =\langle D^*(\phi)\rangle$.

\medskip

\paragraph{\bf Middle quasi-homomorphisms.} We proceed to exhibiting properties of middle quasi-homomorphisms.
The set of all middle quasi-homomorphisms $G\to H$ will be denoted by $\MQHom(G,H)$.

\begin{proposition}
    \label{prop:x-1xy}
    Let $\phi\in \MQHom(G,H)$ and denote $M=M(\phi)$. Then $\phi (x)^{-1}\sim_{M^2}\phi (x^{-1})^{\phi (1)}$ and $\phi (xy)\sim_{M^2M^{-1}}\phi (x)\phi(y)^{\phi (1)}$ for all $x,y\in G$.
\end{proposition}

\begin{proof}
    There exists $t\in M$ such that $\phi(1)=\phi(x^{-1}x)=\phi(x^{-1})t\phi (x)$.
    Therefore, $\phi(x)^{-1}=\phi (1)^{-1}\phi(x^{-1})t=\phi(x^{-1})^{\phi(1)}\phi(1)^{-1}t$. Observe that $\phi(1)^{-1}t\in M^2$, hence the first statement follows.

    As $\phi$ is a middle quasi-homomorphism, we have
    \begin{equation}
        \label{eq:xy3}
        \phi(x)\phi (y^{-1})^{-1}=\phi(xy\cdot y^{-1})\phi(y^{-1})^{-1}\sim _M \phi(xy)
    \end{equation}
    holds true for all $x,y\in G$.
    From the above it follows that $\phi(y^{-1})^{-1}\sim_{M^2}\phi(y)^{\phi(1)}$. Thus we get that 
    \begin{equation}
        \label{eq:xy4}
        \phi(x)\phi(y^{-1})^{-1}\sim_{M^2} \phi(x)\phi(y)^{\phi(1)} 
    \end{equation}   
    for all $x,y\in G$.
    To sum up, the equations \eqref{eq:xy3} and \eqref{eq:xy4} give
    $$\phi(xy)\sim_{M^{-1}}\phi(x)\phi(y^{-1})^{-1}\sim_{M^2}  \phi(x)\phi(y)^{\phi(1)}$$
    for all $x,y\in G$, hence the result.
\end{proof}

The above elementary result has an important consequence:

\begin{corollary}
    \label{cor:unitalm}
    Every unital middle quasi-homomorphism is a quasi-homomorphism.
\end{corollary}

Let $F_2$ be a free group of rank two. 
Inspired by the construction of Brooks' quasi-morphisms \cite{Bro81}, 
Fujiwara and Kapovich found a map $f:F_2\to F_2$ and claimed to be a middle quasi-homomorphism, see  \cite[Theorem 9.6]{FK16}. We note that the map $f$ in question satisfies $f(1)=1$, therefore it would actually need to be a quasi-homomorphism by Corollary \ref{cor:unitalm}. 
In particular, it would follow from Lemma \ref{lem:delta} that $f(F_2)$ is contained in some $R$-neighborhood of $C_H(\Delta_f)$. The latter group is either trivial or infinite cyclic.
On the other hand, it is shown in \cite[Theorem 9.6]{FK16} that the image of $f$ is infinite and is not contained in any $R$-neighborhood of any infinite cyclic subgroup of $F_2$. By Lemma \ref{lem:delta}, the map $f$ thus cannot be a quasi-homomorphism. The contradiction apparently comes from a slightly inaccurate calculation of the word $s_2$ in the proof of \cite[Theorem 9.6]{FK16}.

We proceed by showing that the middle quasi-homomorphisms behave well when being perturbed by a constant. 

\begin{proposition}
    \label{prop:qhomperturb}
    Let $\phi\in\MQHom(G,H)$. For $c\in H$ define the map $\phi_c:G\to H$ by $\phi_c(g)=\phi(g)c$. Then $\phi_c\in\MQHom(G,H)$.
\end{proposition}

\begin{proof}
    Let $x,y\in G$. Then we have that
    $$\phi_c(x)^{-1}\phi_c(xy)\phi_c(y)^{-1}=c^{-1}\phi(x)^{-1}\phi(xy)\phi(y)^{-1}.$$ This shows that $M(\phi_c)\subseteq c^{-1}M(\phi)$.
\end{proof}

Note that the middle quasi-homomorphisms are also invariant under horizontal shifts. That is, if $\phi\in\MQHom(G,H)$, $a\in G$, and a map $\psi:G\to H$ is defined by $\psi(g)=\phi(ga)$, then
\begin{align*}
    \psi(x)^{-1}\psi(xy)\psi(y)^{-1} &= \phi(xa)^{-1}\phi(xya)\phi(a^{-1}ya)^{-1}\\
    &\cdot \phi(a^{-1}ya)\phi(ya)^{-1}\phi(a)\cdot \phi(a)^{-1}
\end{align*}
implies that $M(\psi)\subseteq M(\phi)M(\phi)^{-1}\cdot \phi(a)^{-1}$, i.e., $\psi$ is also a middle quasi-homomorphism.

From Proposition \ref{prop:qhomperturb} and Corollary \ref{cor:unitalm} we immediately obtain the following:

\begin{corollary}
    \label{cor:mq2q}
    Let $\phi\in\MQHom(G,H)$ and $\epsilon=\phi(1)^{-1}$. Then $\phi_\epsilon\in \QHom(G,H)$.
\end{corollary}

The following lemma is a key for studying constant perturbations of quasi-homomorphisms:

\begin{lemma}
    \label{lem:boundconj}
    Let $\phi\in\QHom(G,H)$. Let $\psi: G\to\Delta_\phi$ be any bounded map. Define the maps $\gamma:G\to H$ and $\tilde{\gamma}:G\to H$  by $\gamma(y)=\psi(y)^{\phi(y)}$ and $\tilde{\gamma}(y)=\psi(y)^{\phi(y)^{-1}}$. Then $\lVert\gamma\rVert$ and $\lVert\tilde{\gamma}\rVert$ are bounded in terms of $\phi$ and $\lVert \psi\rVert$ only.
\end{lemma}

\begin{proof}
    Let $D=D(\phi)$. Denote $\lVert\psi\rVert=C$. Let $y\in G$. By Lemma \ref{lem:quasi}, there exist $R>0$, depending only on $\phi$, and $z_y\in C_H(\Delta_\phi)$ such that $d(\phi(y),z_y)\le R$. In other words, we may write $\phi(y)=z_yt$ for some $t$ belonging to the closed $R$-ball $\bar{B}(1,R)$ of 1 in $H$. Then
    $$\gamma(y)=\psi(y)^{z_yt}=\psi(y)^t\in \bar{B}(1,C)^{\bar{B}(1,R)}.$$
    It follows that $\lVert\gamma\rVert$ is bounded in terms of $\phi$ and $C$ only.

    By Lemma \ref{lem:quasi}, we have that $\phi(y)^{-1}=\phi(y^{-1})u$ for some $u\in D^2$. As above, we may write $\phi(y^{-1})=z_{y^{-1}}t$ for some $t\in\bar{B}(1,R)$. Then
    $$\tilde{\gamma}(y)=\psi(y)^{z_{y^{-1}}tu}=\psi(y)^{tu}\in \bar{B}(1,C)^{\bar{B}(1,R)D^2}.$$
    It follows that $\lVert\tilde{\gamma}\rVert$ is bounded in terms of $\phi$ and $C$ only.
\end{proof}

\begin{proposition}
    \label{prop:qcmqhom}
    Let $\phi\in \QHom(G,H)$. Then $\phi_c\in\MQHom(G,H)$ for every $c\in H$.
\end{proposition}

\begin{proof}
    Let $D=D(\phi)$ and pick $c\in H$. Take $x,y\in G$.
    Consider now
    \begin{align*}
        \phi_c(x)^{-1}\phi_c(xy)\phi_c(y)^{-1}
        &= c^{-1}\phi(x)^{-1}\phi(xy)\phi(y)^{-1}\\
        &= c^{-1}\left (\phi(y)^{-1}\phi(x)^{-1}\phi(xy) \right )^{\phi(y)^{-1}}\\
        &= c^{-1}s^{\phi (y)^{-1}}
    \end{align*}
    for $s=\phi(y)^{-1}\phi(x)^{-1}\phi(xy)\in D$. As $D$ is bounded, Lemma \ref{lem:boundconj} implies that there exists a bounded set $T\subseteq H$ such that $s^{\phi (y)^{-1}}\in T$ for all $x,y\in G$. This proves that $M(\phi_c)\subseteq c^{-1}T$, hence the assertion.
\end{proof}

\begin{corollary}
    \label{cor:qhommqhom}
    Every quasi-homomorphism is also a middle quasi-homomorphism.
\end{corollary}

\begin{corollary}
    \label{cor:aff}
    Middle quasi-homomorphisms are precisely constant perturbations of quasi-homomorphisms. 
\end{corollary}

Thus Theorem \ref{thm:qhom} is proved.
As a consequence, every middle quasi-homomorphism $\phi:G\to H$ is constructible in the sense of \cite[Definition 1.1]{FK16}. That is there exist a finite index subgroup $G_o$ of $G$, a subgroup $H_o$ of $H$, a finitely generated abelian subgroup $A$ of $H_o$ with $[H_o,A]=1$, and a quasi-homomorphism $\phi_o:G_o\to H_o$ such that $d(\phi_o,\phi|_{G_o})<\infty$ and the induced map $\bar{\phi_o}:G_o\to H_o/A$ is a homomorphism.

The following three corollaries can be deduced directly from \cite{FK16}:

\begin{corollary}
    \label{cor:FK_4.3}
    Suppose that $\Gamma$ is an irreducible lattice in a semisimple Lie group of real rank $\ge 2$, and let $H$ be a hyperbolic group. Then every middle quasi-homomorphism $\Gamma\to H$ is bounded.
\end{corollary}

\begin{proof}
    This follows from \cite[Corollary 4.3]{FK16}.
\end{proof}

\begin{corollary}
    \label{cor:FK_7.2}
    Suppose that $\Gamma$ is an irreducible lattice in a connected semisimple Lie group of rank $\ge 2$, without compact factors. Let $\Sigma$ be an oriented surface of finite type. Then every middle quasi-homomorphism $\Gamma\to {\rm Map}(\Sigma)$ is bounded.

\end{corollary}

\begin{proof}
    This follows from \cite[Corollary 7.2]{FK16}.
\end{proof}

\begin{corollary}
    \label{cor:FK_8.5}
    Suppose that $\Gamma$ is an irreducible lattice in a connected semisimple Lie group of rank $\ge 2$.
    Suppose that $H$ is a fundamental group of graph of groups where every vertex group is hyperbolic. Then every middle quasi-homomorphism $\Gamma\to H$ is bounded.
\end{corollary}

\begin{proof}
    This follows from \cite[Corollary 8.5]{FK16}.
\end{proof}

\medskip

\paragraph{\bf Middle quasi-homomorphisms and multiplicative quadruples.}
We provide two additional descriptions of middle quasi-homomorphisms. The first one is a direct generalization of a characterization of affine quasi-morphisms exhibited by Tao \cite{Tao18}.
A quadruple $(x_1,x_2,x_3,x_4)\in G^{\times 4}$ is said to be {\it multiplicative} \cite{Gow98,TW07} if we have that $x_1x_2^{-1}x_3x_4^{-1}=1$. The set of all multiplicative quadruples of $G$ will be denoted by
$\mathfrak{M}(G)$.

Given a map $\phi:G\to H$, define
$$A(\phi)=\{ \phi(x_1)\phi(x_2)^{-1}\phi(x_3)\phi(x_4)^{-1}\mid (x_1,x_2,x_3,x_4)\in\mathfrak{M}(G)\}.$$

\begin{lemma}
    \label{lem:Aphi}
    Let $\phi:G\to H$. Denote $A=A(\phi)$ and $M=M(\phi)$. Then $\phi(1)M\subseteq A\subseteq M^{-1}M$.
\end{lemma}

\begin{proof}
    Let $(x_1,x_2,x_3,x_4)\in\mathfrak{M}(G)$. Then 
    \begin{align*}
        \phi(x_1)\phi(x_2)^{-1}\phi(x_3)\phi(x_4)^{-1} &=
    \phi(x_1)\phi(x_2)^{-1}\phi(x_2x_1^{-1})\\
    &\cdot \phi(x_2x_1^{-1})^{-1}\phi(x_3)\phi(x_1x_2^{-1}x_3),
    \end{align*}
    hence the inclusion $A\subseteq M^{-1}M$. For the remaining part observe that $(1,x,xy,y)\in\mathfrak{M}(G)$ for all $x,y\in G$, therefore $\phi(1)\phi(x)^{-1}\phi(xy)\phi(y)^{-1}\in A$. This concludes the proof.
\end{proof}

\begin{corollary}
    \label{cor:taoblog}
    A map $\phi:G\to H$ is a middle quasi-homomorphism if and only if the set
    $A(\phi)$
    is bounded.
\end{corollary}

Note that this proves item (1) of Theorem \ref{thm:multq}.
\medskip

\paragraph{\bf Translates of left quasi-subgroups.}
Our next characterization of middle quasi-homomorphisms is of coarse-geometric nature. Recall \cite{CHT24} that a subset $\Lambda$ of a group $\Gamma$ is called a {\it left quasi-subgroup} of $\Gamma$ if there exist finite subsets $F_1$ and $F_2$ of $\Gamma$ such that $\Lambda^{-1}\subseteq \Lambda F_1$ and $\Lambda^2\subseteq \Lambda F_2$. Unital symmetric left quasi-subgroups are also known as {\it approximate subgroups} \cite{Tao08,Toi20}.

\begin{proposition}
    \label{prop:graph}
    A map $\phi:G\to H$ is a middle quasi-homomorphism if and only if its graph is a right translate of a left-quasi-subgroup of $G\times H$.
\end{proposition}

\begin{proof}
    Let $\phi\in\MQHom(G,H)$. Then the map $\phi^*:G\to H$ defined by $\phi^*(g)=\phi(g)\phi(1)^{-1}$, is a quasi-homomorphism. Therefore the graph $\grph (\phi^*)$ of $\phi^*$ is a left-quasi-subgroup of $G\times H$, see \cite[Lemma 2.37]{CHT24}. It is straightforward to see that $\grph(\phi)=\grph(\phi^*)\cdot (1,\phi(1))$, that is, $\grph(\phi)$ is a right translate of a left-quasi-subgroup of $G\times H$.

    Conversely, let $\grph(\phi)=\Lambda\cdot (a,b)$, where $\Lambda$ is a left-quasi-subgroup of $G\times H$, and $(a,b)\in G\times H$. Then 
    $\Lambda=\grph(\phi)\cdot (a^{-1},b^{-1})=\grph(\Phi)$, where $\Phi:G\to H$ is given by
    $\Phi(g)=\phi(ga)b^{-1}$ for all $g\in G$. By \cite[Lemma 2.37]{CHT24}, the map $\Phi$ is a quasi-homomorphism, and therefore also a middle quasi-homomorphism. As
    \begin{align*}
        \phi(x)^{-1}\phi(xy)\phi(y)^{-1} &= b^{-1}\cdot \Phi(xa^{-1})^{-1}\Phi(xya^{-1})\Phi(aya^{-1})^{-1}\\
        &\cdot \Phi(aya^{-1})\Phi(ya^{-1})^{-1}\Phi(a^{-1})\cdot \Phi(a^{-1})^{-1},
    \end{align*}
    it follows that $M(\phi)\subseteq (M(\Phi)\cdot M(\Phi)^{-1})^b\cdot \phi(1)^{-1}$. Hence $\phi$ is a middle quasi-homomorphism.
\end{proof}





\section{Perturbing a quasi-homomorphism by a constant}
\label{s:perturbing}

\noindent
In this section we consider perturbations of quasi-homomorphisms by constants. We keep the notation $\phi_c$ for the map defined by $\phi_c(x)=\phi(x)c$.

\begin{proposition}
    \label{prop:pertdelta}
    Let $\phi\in \QHom(G,H)$. Pick $c\in\Delta_\phi$. Then $\phi_c$ is a quasi-homomorphism.
\end{proposition}

\begin{proof}
    Let $D=D(\phi)$ and pick $c\in\Delta_\phi$. Take $x,y\in G$. At first observe that
    \begin{align*}
        \phi_c(y)^{-1}\phi_c(x)^{-1}\phi_c(xy)
        &= \left ( \phi(y)^{-1}c^{-1}\phi(x)^{-1}\phi(xy)\right )^c\\
        &= \left ( c^{-\phi(y)}\phi(y)^{-1}\phi(x)^{-1}\phi(xy)\right )^c\\
        &= (c^{-\phi(y)}s)^c
    \end{align*}
    for $s=\phi(y)^{-1}\phi(x)^{-1}\phi(xy)\in D$. By Lemma \ref{lem:boundconj},
    there exists a bounded set $T\subseteq H$ such that  $c^{-\phi(y)}\in T$ for all $y\in G$. It follows that 
    $D(\phi_c)\subseteq (TD)^c$, hence $D(\phi_c)$ is bounded. This shows that $\phi_c\in\QHom(G,H)$, as required.
\end{proof}

Let $\Gamma$ be a group, and $A$ and $B$ subsets of $\Gamma$. We say \cite{Hru12,CHT24} that $A$ and $B$ are {\it left-commensurable} if there exist finite subsets $F_1$ and $F_2$ of $\Gamma$ such that $A\subseteq BF_1$ and $B\subseteq AF_2$. Right-commensurability is defined in an analogous way. The {\it left commensurator} of $A$ in $\Gamma$ is defined by $\LComm_\Gamma(A)=\{ g\in G\mid \exists F\subseteq \Gamma \hbox{ finite with } gAg^{-1}\subseteq AF\}$. The {\it right-commensurator} of $A$ in $\Gamma $ is defined by $\RComm_\Gamma(A)=(\LComm_\Gamma(A^{-1}))^{-1}$, and the {\it commensurator} of $A$ in $\Gamma$ is given by $\Comm_\Gamma(A)=\LComm_\Gamma(A)\cap \RComm_\Gamma(A)$.

\begin{proposition}
    \label{prop:qhomcharacterization}
    Let $\phi\in\MQHom(G,H)$. Then 
    The sets $D(\phi)$ and $(\phi(1)^{-1})^{\phi(G)}$ are left-commensurable.
    In particular, a middle quasi-homomorphism $\phi$ is a quasi-homomorphism if and only if the set $(\phi(1)^{-1})^{\phi(G)}$ is bounded.
\end{proposition}

\begin{proof}
    Let $M=M(\phi)$ and $D=D(\phi)$.
    We have
    \begin{align*}
        \phi(y)^{-1}\phi(x)^{-1}\phi(xy)&=(\phi(1)^{-1})^{\phi(y)}\\
        &\cdot \phi(y)^{-1}\phi(1)\phi(y^{-1})^{-1}\\
        &\cdot \phi(y^{-1})\phi(x)^{-1}\phi(xy),
    \end{align*}
 hence we conclude that $D\subseteq (\phi(1)^{-1})^{\phi(G)}MM^{-1}$ and
 $(\phi(1)^{-1})^{\phi(G)}\subseteq DMM^{-1}$. This concludes the proof.
\end{proof}

\begin{proposition}
    \label{prop:centr}
    Let $\phi\in\MQHom(G,H)$. Suppose there exists $a\in G$ with $\phi(a)\in C_H(\phi(G))$. Then $\phi$ is a quasi-homomorphism if and only if the restriction of $\phi$ to the conjugacy class of $a$ in $G$ is bounded.
\end{proposition}

\begin{proof}
    Let $M=M(\phi)$. We have
    \begin{align*}
        (\phi(1)^{-1})^{\phi(x)} &=\phi(x)^{-1}\phi(a)\phi(x^{-1}a)^{-1}\\
        &\cdot
        \phi(x^{-1}a)\phi(1)^{-1}\phi(a^{-1}x)\\
        &\cdot
        \phi(a^{-1}x)^{-1}\phi(x)\phi(x^{-1}ax)^{-1}\\
        &\cdot \phi(x^{-1}ax)\phi(a)^{-1},
        \end{align*}
    hence the proof follows from Proposition \ref{prop:qhomcharacterization}.
\end{proof}

\medskip

\paragraph{\bf The group of perturbing constants.}
Let $\phi:G\to H$ be a quasi-homomorphism. Given $c\in H$, we are interested in when $\phi_c$ is again a quasi-homomorphism. By Proposition \ref{prop:qhomcharacterization}, this happens precisely when the function $x\mapsto (\phi_c(1)^{-1})^{\phi_c(y)}$ is bounded. As $\phi$ is a quasi-homomorphism, the boundedness condition is equivalent to $x\mapsto (c^{-1})^{\phi(x)}$ being a bounded map. Define
$$\Pi_\phi =\{ c\in H\mid x\mapsto (c^{-1})^{\phi(x)} \hbox{ is bounded}\}.$$
Then it is easy to see that $\Pi_\phi$ is a subgroup of $H$ that contains $\Delta_\phi$ by Proposition \ref{prop:pertdelta}. Note, for instance, that if $\phi$ is bounded, or if $H$ is abelian, then $\Pi_\phi=H$.

The following generalizes Lemma \ref{lem:delta}:

\begin{proposition}
    \label{prop:normPi}
    Let $\phi\in \QHom(G,H)$. Then $\phi(G)$ is contained in $N_H(\Pi_\phi)$.
\end{proposition}

\begin{proof}
    Denote $D=D(\phi)$.
    Let $c\in\Pi_\phi$ and $g\in G$. For arbitrary $x\in G$ we have that
    $$((c^{\phi(g)})^{-1})^{\phi(x)}=(c^{-1})^{\phi(g)\phi(x)}=(c^{-1})^{\phi(gx)t}$$
    for some $t\in D^{-1}$. This proves that $c^{\phi(g)}\in\Pi_\phi$. On the other hand, we also have that 
    $$((c^{\phi(g)^{-1}})^{-1})^{\phi(x)}=(c^{-1})^{\phi(g)^{-1}\phi(x)}=(c^{-1})^{\phi(g^{-1})s\phi(x)}=(c^{-1})^{\phi(g^{-1}x)us^{\phi(x)}}$$
    for some $s\in D$ and $u\in D^{-1}$. By Lemma \ref{lem:boundconj}, we conclude that  $c^{\phi(g)^{-1}}\in\Pi_\phi$.
\end{proof}

\begin{proposition}
    \label{prop:imageinPi}
    Let $\phi\in\QHom(G,H)$. Then $\phi(b)\in\Pi_\phi$ if and only the restriction of $\phi$ to the conjugacy class of $b^{-1}$ in $G$ is bounded.
\end{proposition}

\begin{proof}
    Let $D=D(\phi)$, $D^*=D^*(\phi)$ and $M=M(\phi)$. As $\phi$ is a quasi-homomorphism, these three sets are bounded. We have
    \begin{align*}
        (\phi (b)^{-1})^{\phi(x)} &= 
        \phi(x)^{-1}\phi(b)^{-1}\phi(bx)\\
        &\cdot 
        \phi(bx)^{-1}\phi(x)\phi(x^{-1}b^{-1}x)^{-1}\\
        &\cdot
        \phi(x^{-1}b^{-1}x).
    \end{align*}
    Note that $\phi(x)^{-1}\phi(b)^{-1}\phi(bx)\in D^*$ and $\phi(bx)^{-1}\phi(x)\phi(x^{-1}b^{-1}x)^{-1}\in M$.
    We conclude that the sets $(\phi(b)^{-1})^{\phi(G)}$ and $\phi(\Conj_G(b^{-1}))$ are right commensurable.
    Thus the set $(\phi(b)^{-1})^{\phi(G)}$ is bounded if and only if the restriction of $\phi$ to the conjugacy class of $b^{-1}$ is bounded. This proves the result.
\end{proof}

\begin{corollary}
    \label{cor:fin-by-ab}
    Let $\phi\in \QHom(G,H)$. Suppose that $G$ is finite-by-abelian. Then $\phi(G)\subseteq \Pi_\phi$.
\end{corollary}

\begin{proof}
    If $G$ is finite-by-abelian, then $G'$ is finite. Hence $G$ is a BFC-group, i.e., the cardinalities of conjugacy classes in G are bounded. So the result follows from Proposition \ref{prop:imageinPi}.
\end{proof}

The next result completely describes the group of perturbing constants in the case when the target group is torsion-free hyperbolic:

\begin{proposition}
    \label{prop:hyperbolic}
    Let $H$ be a torsion-free hyperbolic group and $\phi\in\QHom(G,H)$ be unbounded. Suppose that $\Pi_\phi\neq 1$. Then $\phi$ maps $G$ into a cyclic subgroup $C$ of $H$ and $\Pi_\phi=C$.
\end{proposition}

\begin{proof}
    Let $c\in H\setminus\{ 1\}$ and let $\phi_c:G\to H$ be defined by $\phi_c(g)=\phi(g)c$. Suppose $c\in\Pi_\phi\setminus\{ 1\}$. Then $\phi_c$ is an unbounded quasi-homomorphism, and $0<d(\phi,\phi_c)<\infty$. By \cite[Theorem 4.4]{FK16}, the maps $\phi$ and $\phi_c$ map into the same cyclic subgroup $C_c$ of $H$. This implies that $\phi(G)\subseteq C_c\cup C_c\cdot c^{-1}$, the latter being non-empty if and only if $c\in C_c$. We conclude that $\Pi_\phi$ is contained in the union $C$ of the collection $\mathcal{C}$ of all cyclic subgroups that contain $\phi(G)$. As $\phi$ is unbounded, the intersection of the collection $\mathcal{C}$ contains an element $k$ of infinite order. if $C_i=\langle c_i\rangle$ is an arbitrary member of $\mathcal{C}$, then $c_i\in C_H(k)$, which is cyclic. This shows that $C$ is cyclic. Conversely, every element of $C$ is obviously contained in $\Pi_\phi$, hence the result.
\end{proof}

For example, if $F$ is a non-abelian free group and $\phi:G\to F$ is an unbounded quasi-homomorphism whose image is non-commutative, then any non-trivial constant perturbation of $\phi$ yields an example of a middle quasi-homomorphism that is not a quasi-homomorphism.

\medskip
\paragraph{\bf Constant perturbations of graphs.}
Let $\phi:G\to H$ be a quasi-homomorphism. Then \cite[Lemma 2.37]{CHT24} we have that $\grph(\phi)$ is a left-quasi-subgroup of $G\times H$. Define
$$\Pi^*_\phi =\{a,b)\in G\times H\mid (\grph(\phi))\cdot (a,b) \hbox{ is a quasi-left-subgroup of } G\times H\}.$$

\begin{proposition}
    \label{prop:pistar}
    Let $\phi\in \QHom(G,H)$. Let $(a,b)\in G\times H$. Then $(a,b)\in\Pi^*_\phi$ if and only if the set 
    $$S_{(a,b)}(\phi) =\{ \phi(a^{-z})b^{\phi(z)}\mid z\in G\}$$
    is bounded.
\end{proposition}

\begin{proof}
    Denote $D=D(\phi)$.
    Note that $\grph(\phi)\cdot (a,b)=\grph (\Phi)$, where $\Phi:G\to H$ is given by $\Phi(g)=\phi(ga^{-1})b$. Thus $(a.b)\in \Pi^*_\phi$ if and only if $\Phi$ is a quasi-homomorphism. Now observe that
    \begin{align*}
        \Phi(y)^{-1}\Phi(x)^{-1}\Phi (xy) &=
        b^{-1}\cdot b^{-\phi(ya^{-1})}\\
        &\cdot \phi(ya^{-1})^{-1}\phi(xa^{-1})\phi(xa^{-1}ya^{-1})\\
        &\cdot \phi(xa^{-1}ya^{-1})^{-1}\phi(xya^{-1})\phi(a^{-ya^{-1}})\\
        &\cdot \phi(a^{-ya^{-1}})^{-1}\cdot b.
    \end{align*}
Write $$t=\phi(ya^{-1})^{-1}\phi(xa^{-1})\phi(xa^{-1}ya^{-1})\phi(xa^{-1}ya^{-1})^{-1}\phi(xya^{-1})\phi(a^{-ya^{-1}})$$
and note that $t\in (D^{-1})^2$. Then
$$(\Phi(y)^{-1}\Phi(x)^{-1}\Phi (xy))^{b^{-1}}= b^{-\phi(ya^{-1})}\phi(a^{-ya^{-1}})^{-1}\cdot t^{\phi(a^{-ya^{-1}})^{-1}}.$$
Using Lemma \ref{lem:delta} and renaming $z=ya^{-1}$, we readily see that
$S_{(a,b)}(\phi)$ and $D(\Phi)^{b^{-1}}$ are left commensurable sets. Hence
the map $\Phi$ is a quasi-homomorphism if and only if the set $S_{(a,b)}(\phi)$ is bounded. This proves the result.
\end{proof}

\begin{proposition}
    \label{prop:transL}
    Let $\Gamma$ be a group and $\Lambda$ a left quasi-sungroup of $\Gamma$. Let $a\in \Gamma$.
    \begin{enumerate}
        \item if $a\in\Comm_\Gamma (\Lambda)$ then $\Lambda a$ is also a left quasi-subgroup of $\Gamma$.
        \item If $\Lambda$ is symmetric and $\Lambda a$ is a left quasi-subgroup of $\Gamma$, then $a\in \Comm_\Gamma (\Lambda)$.
    \end{enumerate}
\end{proposition}

\begin{proof}
    Let $a\in\Comm_\Gamma (\Lambda)$. Since $a\in\RComm_\Gamma(\Lambda)$, we get that
    $(\Lambda a)^{-1}=a^{-1}\Lambda^{-1}\subseteq\Lambda^{-1} F_1\subseteq (\Lambda a)F_1'$ for some finite subsets $F_1$ and $F_1'$ of $\Gamma$. Because of $a\in\LComm_\Gamma(\Lambda)$ we also get that 
    $(\Lambda a)^2=\Lambda (a\Lambda) a\subseteq \Lambda^2F_2(\Lambda a)F_2'$ or some finite subsets $F_2$ and $F_2'$ of $\Gamma$. This shows (1).

    Now assume that $\Lambda$ is a symmetric left quasi-subgroup, and $\Lambda a$ is a left quasi-subgroup of $\Gamma$. Then $\Lambda a\Lambda\subseteq \Lambda F_1$ for some finite $F_1\subseteq \Gamma$, therefore $a\Lambda\subseteq \Lambda^{-1}\Lambda F_1=\Lambda^2F_1\subseteq \Lambda F_1'$ for some finite $F_1'\subseteq \Gamma$. This shows that $a\in\LComm_\Gamma(\Lambda)$. Similarly, $a^{-1}\Lambda=(\Lambda a)^{-1}\subseteq \Lambda aF_2$ for some finite subset $F_2$ of $\Gamma$, therefore $a\in\RComm_\Gamma(\Lambda)$.
\end{proof}

\begin{corollary}
    \label{prop:commgraph}
    If $\phi:G\to H$ is a symmetric quasi-homomorphism, then $\Comm_{G\times H}\grph(\phi)=\Pi^*_\phi$.
\end{corollary}

If $\phi:G\to H$ is a quasi-homomorphism, then $c\in \Pi_\phi$ if and only if $(1,c)\in\Pi^*_\phi$. Thus Theorem \ref{thm:Piagain} is proved.


\section{Quasi-polynomials of groups}
\label{s:qpoly}

\noindent
\paragraph{\bf Polynomials and quasi-polynomials.}
The notion of middle quasi-homomorphisms can be generalized via the polynomial maps of groups. The latter were introduced by Leibman \cite{Lei98,Lei02} and subsequently further explored by Jamneshan and Thom \cite{JT24}. Given a map $\phi:G\to H$ and $g\in G$, we define a map $\mathfrak{d}_g\phi:G\to H$ by the rule $$(\mathfrak{d}_g\phi)(x)=\phi(gx)\phi(x)^{-1}$$ for all $x\in G$. Then polynomials can be defined inductively as follows. The map $\phi$ is polynomial of degree $-1$ if $\phi(x)=1$ for all $x\in G$. For $d\ge -1$, the map $\phi$ is a polynomial of degree $d+1$ if all the maps $\mathfrak{d}_g\phi$, where $g$ ranges through $G$, are polynomials of degree $d$. The set of all polynomials $\phi:G\to H$ of degree $d$ is denoted by $\Poly_d(G,H)$. It is not difficult to see that polynomials of degree $1$ are precisely constant perturbations of homomorphisms, cf. \cite[Lemma 2.1]{JT24}. 

To obtain a coarse version of polynomials, we follow the ideas of \cite{JT24}. 
Given $g_1,g_2,\ldots ,g_r\in G$, denote
$\mathfrak{d}_{g_1,g_2,\ldots,g_r}=\mathfrak{d}_{g_1}\circ\mathfrak{d}_{g_2}\circ\cdots\circ\mathfrak{d}_{g_r}.$ 
For $d\ge 0$ and $\phi:G\to H$ we set
$$P_d(\phi)=\{ (\mathfrak{d}_{g_1,g_2,\ldots,g_{d+1}}\phi)(1)\mid g_1,g_2,\ldots ,g_{d+1}\in G\}.$$ 
Note that the identity $(\mathfrak{d}_g\phi)(x)=(\mathfrak{d}_{gx}\phi)(1)(\mathfrak{d} _x\phi)(1)^{-1}$ implies that $\phi\in\Poly_d(G,H)$ if and only if $P_d(\phi)=1$.
Accordingly, we say that $\phi$ is a {\it quasi-polynomial} of degree $d$ if the set $P_d(\phi)$ is bounded. The set of all quasi-polynomials of degree $d$ mapping $G$ to $H$ will be denoted by $\QPoly_d(G,H)$. 

By definition, we have that 
$$P_{d+1}(\phi)=\bigcup_{g\in G}P_d(\mathfrak{d}_g\phi).$$
This implies that $\phi\in\QPoly_{d+1}(G,H)$ if and only if all $\mathfrak{d}_g\phi$ are uniformly quasi-polynomials of degree $d$, in the sense that all their middle defect sets are contained in the same finite set.

\begin{lemma}
    \label{lem:circ}
    Let $\phi\in\QPoly_d(G,H)$, and let $q:H\to L$ be a quasi-homomorphism. Then $q\phi\in \QPoly_d(G,L)$.
\end{lemma}

\begin{proof}
    Let $w(x_1,\ldots ,x_k)$ be an arbitrary word in a free group. Then induction on the length $|w|$ of $w$ shows that $$w(q(h_1),q(h_2),\ldots ,q(h_k))\in q(w(h_1,h_2,\ldots ,h_k))F^{\mathcal{O}(|w|)},$$
    where $F=D(q)\cup D(q)^{-1}$ and $|w|$.
    Choose $g_1,g_2,\ldots ,g_{d+1}\in G$.  Then 
    $$(\mathfrak{d}_{g_1,g_2,\ldots ,g_{d+1}}(q\phi))(1)\in q((\mathfrak{d}_{g_1,g_2,\ldots ,g_{d+1}}\phi)(1))F^{\mathcal{O}(2^d)}$$
    implies that $P_d(q\phi)\subseteq q(P_d(\phi))F^{\mathcal{O}(2^d)}$, and this concludes the proof.
\end{proof}

\begin{proposition}
    \label{prop:poly1}
    We have that $\QPoly_0(G,H)$ is precisely the set of all bounded maps $G\to H$, and that $\QPoly_1(G,H)=\MQHom(G,H)$.
\end{proposition}

\begin{proof}
    As $P_0(\phi)=\phi(G)\phi(1)^{-1}$, we clearly get the first part. For the second part, let $\phi:G\to H$, and let $\phi^*$ be defined by $\phi^*(x)=\phi(x)\phi(1)^{-1}$. Then it is straightforward to see that $P_1(\phi)=(D^*(\phi^*))^{-1}$. This shows that $\phi\in\QPoly_1(G,H)$ if and only if $\phi^*$ is a quasi-homomorphism. By Corollary \ref{cor:mq2q}, this is equivalent to the fact that $\phi$ is a middle quasi-homomorphism.
\end{proof}

From here on we will focus exclusively on (quasi)-polynomials of degree $2$, which we also call {\it (quasi)-quadratic maps}. With a little effort, most of the results in what follows can be stated and proved for polynomial maps of arbitrary degree, but we stick with the quadratic case for simplicity. We will repeatedly use the following identity which can be verified by straightforward calculation:

\begin{lemma}
    \label{lem:dg1g2g3}
    Let $\phi:G\to H$ be any map. Then
    \begin{multline*}
        (\mathfrak{d}_{g_1,g_2,g_3}\phi)(1)=\phi(g_3g_2g_1)\phi(g_2g_1)^{-1}\phi(g_1)\phi(g_3g_1)^{-1}\\\phi(g_3)\phi(1)^{-1}\phi(g_2)\phi(g_3g_2)^{-1}$$
    \end{multline*}
    for all $g_1,g_2,g_3\in G$.
\end{lemma}

As constant perturbations of (quasi)-polynomials are evidently (quasi)-polynomials of the same degree, we will only consider unital maps. 

\medskip

\paragraph{\bf Multiplicative quadruples.}
As in Section \ref{s:middle}, we let $\mathfrak{M}(G)$ to be the set of multiplicative quadruples of $G$. Note that $G$ acts on $\mathfrak{M}(G)$ from the right by component-wise multiplication. Given $\phi:G\to H$, we have a map $\mu_\phi :\mathfrak{M}(G)\to A(\phi)$ given by $\mu_\phi(\mathbf{x})=\phi(x_1)\phi(x_2)^{-1}\phi(x_3)\phi (x_4)^{-1}$ for $\mathbf{x}=(x_1,x_2,x_3,x_4)\in\mathfrak{M}(G)$.

\begin{proposition}
\label{prop:equiv}
Let $\phi:G\to H$ be a unital map.
\begin{enumerate}
    \item $\phi$ is quadratic if and only if the map $\mu_\phi$ is equivariant, that is, $\mu_\phi(\mathbf{x}\cdot t)=\mu_\phi(\mathbf{x})$ for all $\mathbf{x}\in\mathfrak{M}(G)$ and all $t\in G$.
    \item $\phi$ is quasi-quadratic if and only if $\mu_\phi$ is equivariant up to bounded error, that is, there exists $C>0$ such that $d(\mu_\phi(\mathbf{x}\cdot t),\mu_\phi(\mathbf{x}))\le C$ for all $\mathbf{x}\in\mathfrak{M}(G)$ and all $t\in G$.
\end{enumerate}
\end{proposition}

\begin{proof}
    First we deal with (1).
    Let $\phi$ be a unital quadratic map, $\mathbf{x}=(x_1,x_2,x_3,x_4)\in \mathfrak{M}(G)$ and $t\in G$. Apply Lemma \ref{lem:pol2} with $g_1=x_3$, $g_2=x_2x_3^{-1}$ and $g_3=x_4x_3^{-1}$ to derive 
    $$
        \phi(x_1)\phi(x_2)^{-1}\phi(x_3)\phi(x_4)^{-1}=\phi(x_1x_3^{-1})\phi(x_2x_3^{-1})^{-1}\phi(x_4x_3^{-1})^{-1}.
    $$
    This can be rewritten as $\mu_\phi(\mathbf{x})=\mu_\phi(\mathbf{x}\cdot x_3^{-1})$. As this is true for all $\mathbf{x}\in \mathfrak{M}(G)$, we may apply it with $\mathbf{x}\cdot t$ to obtain $\mu_\phi(\mathbf{x}\cdot t)=\mu_\phi((\mathbf{x}\cdot t)\cdot (x_3t)^{-1})=\mu_\phi(\mathbf{x}\cdot x_3^{-1})=\mu_\phi(\mathbf{x})$. This shows that $\mu_\phi$ is equivariant.

    Conversely, suppose that the map $\mu_\phi$ is equivariant. Pick $g_1,g_2,g_3\in G$. Then $\mathbf{x}=(g_3g_2g_1,g_2g_1,g_1,g_3g_1)$ is a multiplicative quadruple. Therefore $\mu_\phi(\mathbf{x})=\mu_\phi(\mathbf{x}\cdot g_1^{-1})$. It is easy to see that this is equivalent to $(\mathfrak{d}_{g_1,g_2,g_3}\phi)(1)=1$.

    The proof of (2) follows the above lines. At first we introduce some notation. If $\mathbf{x}=(x_1,x_2,x_3,x_4)\in \mathfrak{M}(G)$, then we denote $\mathbf{x}^{\rm opp}=(x_4,x_3,x_2,x_1)$. Clearly, $\mathbf{x}^{\rm opp}\in\mathfrak{M}(G)$ and $\mu_\phi(\mathbf{x}^{\rm opp})=\mu_\phi(\mathbf{x})^{-1}$. Now assume that $\phi$ is quasi-quadratic. Denote $S=P_2(\phi)$. Then the above argument can be adjusted to conclude that $\mu_\phi(\mathbf{x}\cdot t)\mu_\phi(\mathbf{x}\cdot x_3^{-1})^{-1}\in S$ for all $\mathbf{x}\in\mathfrak{M}(G)$ and all $t\in G$. This can clearly be rewritten as $\mu_\phi(\mathbf{x}^{\rm opp}\cdot t)\sim_S\mu_\phi(\mathbf{x}^{\rm opp}\cdot x_3^{-1})$. When $\mathbf{x}$ run through all $\mathfrak{M}(G)$, the same goes for $\mathfrak{x}^{\rm opp}$, and this readily implies that $\mu_\phi$ is equivariant up to bounded error. The converse is similar.
\end{proof}

\medskip

\paragraph{\bf Normal quasi-quadratic maps.}
Let $\phi\in\QPoly_2(G,H)$ be a unital quasi-quadratic map. We say that $\phi$ is {\it nearly normal} if there exist finite subsets $K$ and $Y$ in $H$ such that
\begin{enumerate}
    \item[(Q1)] $P_2(\phi)\subseteq K$,
    \item[(Q2)] $\phi(G)\subseteq N_H(\Xi_\phi)$, where $\Xi_\phi =\langle K\rangle$,
    \item[(Q3)] $Y$ is a symmetric set, and $\phi(G)\cup\phi(G)^{-1}$ is covered by the union of cosets $C_H(\Xi_\phi)y$, where $y\in Y$.   
\end{enumerate}

If $\phi:G\to H$ is a nearly normal quasi-quadratic map, then $\Xi_\phi$ is normal in $\langle \phi(G)\rangle$ by (Q2). By replacing $H$ with $\langle\phi(G)\rangle$ we may, and will from here on, assume that $\phi(G)$ generates $H$, and that $\Xi_\phi$ is normal in $H$. Furthermore, the induced map $G\to H/\Xi_\phi$ is a unital quadratic map. Note also that (Q3) is essentially the same as saying that the set $K^{\phi(G)\cup\phi(G)^{-1}}$ is finite.

If a nearly normal quadratic map $\phi$ additionally satisfies
\begin{enumerate}
    \item[(Q4)] There exists $n\ge 0$ such that $Y^{n+1}\subseteq Y^nC_H(\Xi_\phi)\Xi_{\phi}$,
\end{enumerate}
then we say that $\phi$ is a {\it normal quasi-quadratic map}.

Among examples of normal quasi-quadratic maps we find unital bounded maps, unital quasi-homomorphisms (see Lemma \ref{lem:delta}), and unital quasi-quadratic maps with commutative targets.

We define a map $\phi:G\to H$ to be {\it almost quadratic} if $P_2(\phi)$ is contained in a finite normal subgroup $N$ of $H$. Such a map satisfies (Q1)--(Q4). One can take $K=N$, and $Y$ comes from the fact that $H/C_H(N)$ is a finite group. Thus unital almost quadratic maps are normal.

One can also show:

\begin{proposition}
    \label{prop:corcnormal}
    Let $\phi:G\to H$ be a unital quadratic map, and let $q:H\to L$ be a unital quasi-homomorphism. Then $q\phi$ is a normal quasi-quadratic map.
\end{proposition}

\begin{proof}
    Let $F=D(q)\cup D(q)^{-1}$. Then the proof of Lemma \ref{lem:circ} shows that $P_2(q\phi)\subseteq F^n$ for some $n$. Set $K=F^n$ and $\Xi_{q\phi}=\langle K\rangle=\Delta_q$. These meet the conditions (Q1) and (Q2) by Lemma \ref{lem:delta}. The set $Y$ meeting the conditions (Q3) and (Q4) for $q\phi$ can now be obtained from part (3) of Lemma \ref{lem:delta}.
\end{proof}

\medskip

\paragraph{\bf  Constructibility of normal quasi-quadratic maps.}
 We recall a construction from \cite{JT24}. Given a group $G$, let $\Pol2 G$ be a group generated by the symbols $\tau(g)$, where $g\in G$, subject to the relations
$$\tau(g_3g_2g_1)\tau(g_2g_1)^{-1}\tau(g_1)\tau(g_3g_1)^{-1}\tau(g_3)\tau(g_2)\tau(g_3g_2)^{-1}=1$$
for all $g_1,g_2,g_3\in G$. The structure of $\Pol2 G$ is described in \cite[Theorem 1.3]{JT24}. We will use the following:

\begin{lemma}[\cite{JT24}]
    \label{lem:pol2}
    There is an epimorphism $\pi :\Pol2 G\to G$, whose kernel $N(G)$ is an abelian group. For every unital quadratic map $\phi :G\to H$, there exists a unique homomorphism $\kappa :\Pol2 G\to H$ such that $\kappa\tau=\phi$.
\end{lemma}

We now explain how normal quasi-quadratic maps can be constructed in the spirit of constructibility of quasi-homomorphisms obtained in \cite[Theorem 3.6]{FK16}. 

\begin{proof}[Proof of Theorem \ref{thm:normalpoly}]
    Let $\phi:G\to H$ be a normal quasi-quadratic map, and let $K$, $\Xi_\phi$ and $Y$ be as in (Q1)--(Q4). As $\Xi_\phi$ is normal in $H$, we have automorphisms $\ad (\phi (g))$ and $\ad (\phi (g)^{-1})$ of $\Xi_\phi$, where $g\in G$. By the condition (Q3), for every $g\in G$ there exist $y,y'\in Y$, depending on $g$, such that $\ad(\phi(g))=\ad y$ and $\ad (\phi (g)^{-1})=\ad y'$. We therefore conclude that the maps $\tilde{\Phi}^{\pm} :G\to \Aut\Xi_\phi$, given by $\tilde{\Phi}^{\pm}(g)=\ad (\phi(g)^{\pm 1})$, have finite images.

    As $P_2(\phi)$ is contained in $\Xi_\phi$, the map $\tilde{\Phi}^+$ induces a unital quadratic map $\Phi:G\to\Out\Xi_\phi$. The map $\tilde{\Phi}^-$ gives rise to a map $\Psi:G\to \Out\Xi_\phi$, and we have that $\Psi(g)=\Phi(g^{-1})$. By Lemma \ref{lem:pol2}, there is a unique homomorphism $\kappa:\Pol2 G\to\Out\Xi_\phi$ such that $\kappa\tau=\Phi$. Denote $G_o=\ker\kappa$.
    We claim that $G_o$ has finite index in $\Pol2 G$. Pick arbitrary $g_1,\ldots ,g_r\in G$ and $\epsilon_i=\pm 1$ for $i=1,\ldots ,r$. Note that
    $$\kappa \left ( \prod_i^r \tau(g_i)^{\epsilon_i}\right )=\ad\left (\prod_i^r \phi(g_i)^{\epsilon_i}\right )\Inn \Xi_\phi=\ad\left (\prod_i^r y_{j_i}^{\epsilon_i}\right )\Inn \Xi_\phi.$$
    By the condition (Q4), the homomorphism $\kappa$ has finite image, hence the claim follows.

    Given $z\in \Pol2(G)$, there exists $y\in Y$ such that $\phi\pi(z)\in C_H(\Xi_\phi)y$. If $z\in G_o$, then $\ad \phi\pi(z)\in \Inn\Xi_\phi$. In this case we can therefore choose $y$ to belong to $\Xi_\phi$. Hence $\tilde{\Phi}^+\pi(G_0)\cup \tilde{\Phi}^-\pi(G_0)$ is covered by finitely many cosets $C_H(\Xi_\phi)z_i$, where $z_i\in\Xi_\phi$, $i=1,\ldots,s$.

    We now consider the map $\hat{\phi}=\phi\pi:\Pol2 G\to H$ instead of $\phi$. It is fairly straightforward to see that $\hat{\phi}$ is a normal quasi-quadratic map. Namely, we have that $(\mathfrak{d}_{t_1,t_2,t_3}\hat{\phi})(1)=(\mathfrak{d}_{\pi(t_1),\pi(t_2),\pi(t_3)}\phi)(1)$ for all $t_1,t_2,t_3\in\Pol2 G$. As $\pi$ is surjective, we conclude that $P_2(\hat{\phi})=P_2(\phi)$. As $\hat{\phi}(\Pol2 G)^{\pm 1}=\phi(G)^{\pm 1}$, we can take the same $K$ and $Y$ for $\hat{\phi}$ to fulfill (Q1)--(Q4). In particular, $\Xi_{\hat{\phi}}=\Xi_\phi$.

    Mimicking \cite{FK16}, we define $r:\hat{\phi}(G_o)\to C_H(\Xi_\phi)$ by sending $x\in \hat{\phi}(G_o)$ to a chosen element $r(x)$ of $C_H(\Xi_\phi)$ with the property that $r(x)^{-1}x=z_i$ for some $i=1,\ldots ,s$. We claim that the map $\phi_o=r\hat{\phi}: G_0\to C_H(\Xi_\phi)$ is normal quasi-quadratic. Given $z\in G_o$, we have that $\hat{\phi}(z)=\phi_o(z)z_i$ for some $i=1,\ldots ,s$. Observe that $z_i$ commute with all elements of $\phi_o(G_o)$. Choose arbitrary $x_1,x_2,x_3\in G_o$. Then $(\mathfrak{d}_{x_1,x_2,x_3}\hat{\phi})(1)=(\mathfrak{d}_{x_1,x_2,x_3}\phi_o)(1)w$, where $w$ is a word of uniformly bounded length in $z_i^{\pm 1}$, $i=1,\ldots ,s$. Let $S$ be the finite set containing all possible $w$. Then $P_2(\phi_o)\subseteq KS^{-1}$, and $\Xi_{\phi_o}\subseteq \Xi_{\hat{\phi}}=\Xi_\phi$. As $\phi_o(G_o)\cup \phi_o(G_o)^{-1}$ is covered by the coset $C_H(\Xi_{\phi_o})\cdot 1$, the map $\phi_o$ obviously satisfies (Q4).

    It is easy to see that, given $z\in G_o$, we have that 
    $d(\hat{\phi}(z),\phi_o(z))\le \max\limits_{i=1,\ldots, s}\lVert z_i\rVert,$
    hence $$d(\hat{\phi}\big|_{G_o},\phi_o)<\infty.$$ Furthermore, the map $\phi_o$ induces a unital quadratic map $\tilde{\phi}_o:G_o\to \langle \phi_o(G_o)\rangle/\Xi_{\phi_o}$ with the additional property that $\Xi_{\phi_o}$ is central in $\langle \phi_o(G_o)\rangle$. Note that $\Xi_{\phi_o}$ is clearly finitely generated. This finishes the proof of Theorem \ref{thm:normalpoly}.
\end{proof}

We remark that the above proof shows that if $\phi:G\to H$ is a unital quasi-quadratic map with the property that $\phi$ maps $G$ into $C_H(\langle P_2(\phi)\rangle)$, then $\phi$ is automatically normal.
\medskip

\paragraph{\bf Torsion-free hyperbolic targets.}

\noindent
An application of Theorem \ref{thm:normalpoly} is the following quadratic counterpart of \cite[Theorem 4.1]{FK16}:

\begin{proposition}
    \label{prop:normalhyperb}
    Let $H$ be a torsion-free hyperbolic group. Let $\phi:G\to H$ be a nearly normal quasi-quadratic map. If $\phi$ is not bounded, then either $\phi$ is quadratic, or $\phi(G)$ is contained in a cyclic subgroup of $H$.
\end{proposition}

\begin{proof}
    Let $K$ and $Y$ be as in (Q1)--(Q3). Let $\Xi_\phi=\langle K\rangle$. If $\Xi_\phi$ is trivial, the map $\phi$ is quadratic by (Q1). If $\Xi_\phi$ is non-trivial, it has to be infinite. As $\phi(G)$ is infinite, the centralizer $C_H(\Xi_\phi)$ is infinite by (Q3), therefore $\Xi_\phi$ is infinite cyclic \cite[Corollary III.$\Gamma$.3.10]{BH99}. Then it follows from \cite[Proposition III.$\Gamma$.3.16]{BH99} that $N_H(\Xi_\phi)=C_H(\Xi_\phi)$ is infinite cyclic, too, therefore the result follows by (Q2). 
\end{proof}

We proceed now to proving Theorem \ref{thm:distancenearly}.

\begin{lemma}[\cite{Lei98}, Lemma 1.21]
    \label{lem:leibman}
    Let $\phi:G\to H$ be a non-constant quadratic map. If $H$ is torsion-free, then $\phi(G)$ is infinite.
\end{lemma}

\begin{lemma}
    \label{lem:Z}
    Let $\phi :\mathbb{Z}\to H$ be a unital quadratic map.
    \begin{enumerate}
        \item $\phi(\mathbb{Z})\subseteq\langle \phi(1),\phi(2)\rangle$.
        \item $[\phi(1),\phi(2)]$ commutes with $\phi(2)^{-1}\phi(1)^2$.
    \end{enumerate}
\end{lemma}

\begin{proof}
    (1) follows directly from Lemma \ref{lem:pol2} which gives, applied to 
    $(\mathfrak{d}_{1,1,n-1}\phi)(1)=1$, the recursive equation
    $$\phi(n+1)\phi(2)^{-1}\phi(1)\phi(n)^{-1}\phi(n-1)\phi(1)\phi(n)^{-1}=1$$ for all $n\in\mathbb{Z}$.

    Denote $a=\phi(1)$, $b=\phi(2)$. The above equality gives
    \begin{equation}
        \label{eq:Z1}
        \phi(3)=ba^{-2}ba^{-1}b.
    \end{equation}
    Lemma \ref{lem:pol2}, applied to $(\mathfrak{d}_{1,1,-1}\phi)(1)=1$, gives
    \begin{equation}
        \label{eq:Z2}
        \phi(-1)=a^{-1}ba^{-2},
    \end{equation}
    whereas $(\mathfrak{d}_{1,-1,2}\phi)(1)=1$ implies
    \begin{equation}
        \label{eq:Z3}
        \phi(-1)=b^{-1}\phi(3)a^{-1}b^{-1}a.
    \end{equation}
    Comparing \eqref{eq:Z2} and \eqref{eq:Z3} and using \eqref{eq:Z1}, we get
    $ab^{-1}aba^{-2}=ba^{-1}b^{-1}a$, which is the same as $[b,a^{-1}]^{a^{-1}}=[b^{-1},a]$. From here it is easy to derive (2).
\end{proof}

\begin{proof}[Proof of Theorem \ref{thm:distancenearly}]
    Let $H$ be a torsion-free hyperbolic group, and let $\phi_1,\phi_2:G\to H$ be nearly normal quasi-quadratic maps with $d(\phi_1,\phi_2)$ finite. If $\phi_1$ is bounded, then so is $\phi_2$. So assume that both $\phi_1$, $\phi_2$ are unbounded. if $\phi_1$ maps into a cyclic subgroup $C$ of $H$, then the boundary argument of \cite[Proof of Theorem 4.4]{FK16} can be applied to show that $\phi_2(G)\subseteq C$. By Proposition \ref{prop:normalhyperb}, we are thus left with the case when both $\phi_1$ and $\phi_2$ are unbounded unital quadratic maps that do not map $G$ into a cyclic subgroup of $H$. We claim that this implies $\phi_1=\phi_2$.

    Suppose there exists $x\in G$ such that $\phi_1(x)\neq \phi_2(x)$. From here on we restrict $\phi_1$ and $\phi_2$ to the cyclic subgroup $\langle x\rangle$ of $G$. If $\phi_1$ is constant on $\langle x\rangle$, then $\phi_1\equiv 1$. As $\phi_2(x)\neq \phi_1(x)$, it follows from Lemma \ref{lem:leibman} that $\phi_2$ is unbounded, but then $d(\phi_1,\phi_2)=\infty$. This contradiction, together with Lemma \ref{lem:leibman}, shows that both $\phi_1$ and $\phi_2$, when restricted to $\langle x\rangle$, are unbounded. In particular, $\langle x\rangle$ is infinite cyclic, and thus we are dealing with the unital quadratic maps $\phi_1:\mathbb{Z}\to \langle \phi_1(1),\phi_1(2)\rangle$ and $\phi_2:\mathbb{Z}\to \langle \phi_2(1),\phi_2(2)\rangle$ with $\phi_1(1)\neq \phi_2(1)$. Denote $a=\phi_1(1)$, $b=\phi_1(2)$. By assumption $[a,b]\neq 1$. By Lemma \ref{lem:Z} we have that $b^{-1}a^2\in C_H([a,b])$, therefore $b^{-1}a^2$ and $[a,b]$ are both powers of the same element $z\in C_H([a,b])$, that is $b^{-1}a^2=z^k$ and $[a,b]=z^\ell$ for some $k,\ell\in\mathbb{Z}$. Then
    $$z^\ell=[a,a^2z^{-k}]=[a,z^{-k}],$$
    which gives
    \begin{equation}
        \label{eq:H1}
        z^{\ell +k}=(z^{k})^a\in C_H([a,b])\cap C_H([a,b])^a.
    \end{equation}
    If $a\in C_H([a,b])$, then $b\in C_H([a,b])$, hence $a$ and $b$ commute, a contradiction. If $a\notin C_H([a,b])$, then \eqref{eq:H1} implies $k=-\ell$, hence $b=a^2$, and again we conclude that $\langle a,b\rangle$ needs to be cyclic. This final contradiction shows that $\phi_1=\phi_2$.
\end{proof}





\begin{thebibliography}{9}

    \bibitem[BH99]{BH99}
    M. R. Bridson, and A. Haefliger, {\it Metric spaces of non-positive curvature}, Springer-Verlag, Berlin, Heidelberg, 1999.

    \bibitem[Bro81]{Bro81}
    R. Brooks, {\it Some remarks on bounded cohomology}, in Annals of Mathematics Studies, Vol. 97. Princeton University Press, Princeton, 1981.

    \bibitem[CHT24]{CHT24}
    M. Cordes, T. Hartnick, and V. Toni\'{c}, {\it Foundations of geometric approximate group theory}, arXiv:2012.15303v3, 1 Apr 2024. (with an appendix by S. Machado).

    \bibitem[FK16]{FK16} 
    K. Fujiwara, and M. Kapovich,
    {\it On quasihomomorphisms with noncommutative targets},
    Geom. Funct. Anal. {\bf 26} (2016), 478--519.

    \bibitem[Gow98]{Gow98}
    T. Gowers, {\it A new proof of Szemer\' edi's theorem for arithmetic progressions of length four}, Geom. Func. Anal. {\bf 8} (1998), 529-–551.

    \bibitem[HS16]{HS16}
    T. Hartnick, and P. Schweitzer, {\it On quasioutomorphism groups of free groups and their transitivity properties}, J. Algebra {\bf 450} (2016), 242--281.

    \bibitem[Heu19]{Heu19} N. Heuer, {\it Constructions in Stable
    Commutator Length and
    Bounded Cohomology}, PhD thesis, 2019.

    \bibitem[Hru12]{Hru12}
    E. Hrushovski, {\it Stable group theory and approximate subgroups},
    J. Amer. Math. Soc. {\bf 25}(1) (2012), 189--243.

    \bibitem[JT24]{JT24}
    A. Jamneshan, and A. Thom, {\it Quadratic maps between non-abelian groups},
    arXiv 2412.14908v1, 2024.

    \bibitem[Lei98]{Lei98}
    A. Leibman, {\it Polynomial sequences in groups}, J. Algebra {\bf 201} (1998), no. 1, 189–-206.

    \bibitem[Lei02]{Lei02}
    A. Leibman, {\it
    Polynomial mappings of groups}, Israel J. Math. {\bf 129} (2002), 29–-60. 

    \bibitem[Mac23a]{Mac23a}
    S. Machado, {\it Approximate lattices: structure in linear groups, definition(s) and beyond},
    arXiv:2306.09899v1, 2023.

    \bibitem[Mac23b]{Mac23b}
    S. Machado, {\it The structure of approximate lattices in linear groups},
    arXiv:2306.09899v2, 2023.

    \bibitem[Tao08]{Tao08}
    T. Tao, {\it Product set estimates for non-commutative groups}, Combinatorica {\bf 28} (2008), 547–-594.

    \bibitem[Tao18]{Tao18}
    T. Tao, {\it $1\%$ quasimorphisms and group cohomology}, \url{https://terrytao.wordpress.com/2018/07/07/1-quasimorphisms-and-group-cohomology/}. Accessed: May 18, 2025.

    \bibitem[TW07]{TW07}
        T. Tao, and V. Wu, {\it Additive combinatorics}, Cambridge University Press, 2007.

    \bibitem[Toi20]{Toi20}
        M. C. H. Tointon, {\it Introduction to approximate groups},
        London Mathematical Society Student Texts 94, Cambridge University Press, 2020.

    \bibitem[Ula60]{Ula60}
        S. M. Ulam, {\it A collection of mathematical problems}, Interscience Tracts in Pure and Applied Mathematics, vol 8, Interscience Publishers, 

\end{thebibliography}
\end{document}